\documentclass{article}
\usepackage[utf8]{inputenc}
\usepackage{comment}
\usepackage{amssymb,amsmath,amsthm,graphicx,mathtools,enumerate,mathrsfs}
\usepackage{fullpage}

\usepackage[dvipsnames]{xcolor}

\usepackage[
	setpagesize=false,
	colorlinks=true,
	linkcolor=BrickRed,
        citecolor=OliveGreen,
        urlcolor=black,
	pdfencoding=auto,
	psdextra,
]{hyperref}
\usepackage{cleveref}

\usepackage[
	a4paper,
	margin=1in
]{geometry}

\usepackage{setspace}
\usepackage[
	a4paper,
	margin=1in
]{geometry}
\usepackage{setspace}
\usepackage{sectsty} 
\usepackage{titlesec}
\usepackage{appendix}

\usepackage{cleveref}
\theoremstyle{plain}
\newtheorem{theorem}{Theorem}[section]
\crefname{theorem}{Theorem}{Theorems}

\newtheorem{proposition}[theorem]{Proposition}
\crefname{proposition}{Proposition}{Propositions}

\newtheorem{corollary}[theorem]{Corollary}
\crefname{corollary}{Corollary}{Corollaries}

\newtheorem{lemma}[theorem]{Lemma}
\crefname{lemma}{Lemma}{Lemmas}

\newtheorem{conjecture}[theorem]{Conjecture}
\crefname{conjecture}{Conjecture}{Conjectures}

\crefname{problem}{Problem}{Problem}

\newtheorem{claim}{Claim}[theorem]
\newtheorem*{claim*}{Claim}

\makeatletter
\newenvironment{claimproof}[1][Proof]{\par
	\pushQED{\qed}%
	
	\normalfont \topsep6\p@\@plus6\p@\relax
	\trivlist
	\item[\hskip\labelsep
	\textit{#1}\@addpunct{.}~]\ignorespaces
}{%
	\popQED\endtrivlist\@endpefalse
}
\makeatother

\crefname{observation}{Observation}{Observations}

\crefname{setup}{Setup}{Setups}

\crefname{fact}{Fact}{Facts}

\crefname{algorithm}{Algorithm}{Algorithms}

\crefname{remark}{Remark}{Remarks}

\crefname{example}{Example}{Examples}

\theoremstyle{definition}
\newtheorem{definition}[theorem]{Definition}
\crefname{definition}{Definition}{Definitions}

\crefname{construction}{Construction}{Constructions}

\newtheorem{question}[theorem]{Question}
\crefname{question}{Question}{Questions}

\numberwithin{equation}{section}

\usepackage[shortlabels]{enumitem}
\setlist[enumerate,1]{label={\upshape (\roman*)}}

\definecolor{DarkDesaturatedBlue}{HTML}{3A3556}
\definecolor{VividOrange}{HTML}{F15918}
\definecolor{PureOrange}{HTML}{FFBA00}
\definecolor{LightGrayishPink}{HTML}{EEC5D5}
\definecolor{VerySoftBlue}{HTML}{B5AFDB}

\usepackage{tikz}
\usetikzlibrary{math}
\usetikzlibrary{snakes,calc,decorations.pathmorphing,through}
\tikzset{snake it/.style={decorate, decoration=snake}}

\definecolor{DarkDesaturatedBlue}{HTML}{3A3556}
\definecolor{VividOrange}{HTML}{F15918}
\definecolor{PureOrange}{HTML}{FFBA00}
\definecolor{LightGrayishPink}{HTML}{EEC5D5}
\definecolor{VerySoftBlue}{HTML}{B5AFDB}
\newcounter{propcounter}

\begin{document}
	\title{Spanning clique subdivisions in pseudorandom graphs}
	\author{Hyunwoo Lee \thanks{Department of Mathematical Sciences, KAIST and Extremal Combinatorics and Probability Group (ECOPRO), Institute for Basic Science (IBS), Daejeon, South~Korea. Supported by the National Research Foundation of Korea (NRF) grant funded by the Korea government(MSIT) No. RS-2023-00210430, and the Institute for Basic Science (IBS-R029-C4). Email: \texttt{hyunwoo.lee@kaist.ac.kr}}
    \and Mat\'ias Pavez-Sign\'e\thanks{Supported by ANID-FONDECYT Regular grant No. 1241398, by ANID Basal Grant CMM FB210005, and by the European Research Council (ERC) under the European Union Horizon 2020 research and innovation programme (grant agreement No. 947978). Department of Mathematical Engineering and Center for Mathematical Modeling  (CNRS IRL2807), University of Chile, Santiago, Chile. Email: \texttt{mpavez@dim.uchile.cl}}\and Teo Petrov\thanks{Mathematics Institute, University of Warwick, Coventry, United Kingdom. Supported by the Warwick Mathematics Institute Centre for Doctoral Training, and by funding from the UK EPSRC (Grant number: EP/W524645/1). Email: \texttt{teodor.petrov@warwick.ac.uk}}}
	
	\maketitle
	
	
	\begin{abstract}
		In this paper, we study the appearance of a spanning subdivision of a clique in graphs satisfying certain pseudorandom conditions. Specifically, we show the following results. 
		\begin{enumerate}
        	
            	\item  There are constants $C>0$ and $c\in (0,1]$ such that, whenever $d/\lambda\ge C$, every $(n,d,\lambda)$-graph contains a spanning subdivision of $K_t$ for all  $2\le t \le \min\{cd,c\sqrt{\frac{n}{\log n}}\}$.
			\item  There are constants $C>0$ and $c\in (0,1]$ such that, whenever $d/\lambda\ge C\log^3n$, every $(n,d,\lambda)$-graph contains a spanning nearly-balanced subdivision of $K_t$ for all  $2\le t \le \min\{cd,c\sqrt{\frac{n}{\log^3n}}\}$.
		\item For every $\mu>0$, there are constants $c,\varepsilon\in (0,1]$ and $n_0\in \mathbb N$ such that, whenever $n\ge n_0$, every $n$-vertex graph with minimum degree at least $\mu n$ and no bipartite holes of size $\varepsilon n$ contains a spanning nearly-balanced subdivision of $K_t$ for all $2\le t \le c\sqrt{n}$.
		
		\end{enumerate}
	\end{abstract}
	
	
	\section{Introduction}\label{sec:intro}
	
	Given a graph $H$, we say that a graph $H'$ is a \emph{subdivision} of $H$ (or an \emph{$H$-subdivision}) if $H'$ is obtained by replacing one or more edges from $H$ with vertex-disjoint paths. In this paper, we will seek simple conditions on a graph $G$ that force the containment of a subdivision of the complete graph on $t$ vertices, which is denoted by $K_t$. This problem has been studied since the 60s, when Mader~\cite{MADER1967} proved that there is a function $f(t)$ such that every graph with an average degree at least $f(t)$ contains a $K_t$-subdivision. After this influential result, many researchers have been looking for optimal degree conditions forcing the containment of $K_t$-subdivisions. For example, Bollob\'as and Thomason~\cite{BOLLOBAS1998883} and Koml\'os and Szemer\'edi~\cite{KomlosSzemeredi_subdivisions} showed that one can take $f(t)=\Theta(t^2)$ so that every graph with average degree at least $f(t)$ contains a $K_t$-subdivision, settling a conjecture of Mader~\cite{MADER1967} and of Erd\H os and Hajnal~\cite{erdos1964complete}. Another direction of this problem is to ask for the containment of a $K_t$-subdivision which is also \emph{balanced}, where a balanced subdivision of $K_t$ is a graph obtained by replacing every edge of $K_t$ with vertex-disjoint paths of the same length. Solving a conjecture of Thomassen~\cite{Thomassen1984}, Liu and Montgomery~\cite{liu2023solution} proved that having a large average degree is enough to force the containment of a balanced $K_t$-subdivision. Shortly after, Luan, Tang,
	Wang, and Yang~\cite{luan2023balanced} and Gil Fern\'andez, Hyde, Liu, Pikhurko, and Wu~\cite{fernandez2023disjoint} independently proved that average degree  $\Omega(t^2)$ is enough to contain a balanced $K_t$-subdivision.
	
	When the host graph is dense, one typically expects to find even larger subdivisions and thus it is natural to ask for a Dirac-type result for \textit{spanning subdivisions}. Given a graph $H$, one aims to find the smallest number $\delta>0$ such that every $n$-vertex graph $G$ with minimum degree $\delta(G)\ge \delta n$ contains an $H$-subdivision using all the vertices of $G$. For instance, a spanning $K_3$-subdivision is just a Hamilton cycle and thus Dirac's theorem~\cite{D1952} tells us that, for $n\ge 3$, minimum degree at least $\frac{n}{2}$ is enough. In general, for $t=O(\sqrt{n})$, Pavez-Sign\'e~\cite{pavezsigne_2023} proved that every $n$-vertex graph with minimum degree at least $(1+o(1))\frac{n}{2}$ contains a spanning $K_t$-subdivision, which is also nearly-balanced \footnote{An $H$-subdivision is nearly-balanced if there is positive integer $\ell$ so that each path between branch vertices has length between $\ell$ and $\ell+1$.}. See~\cite{cheng2024spanning,lee2023spanning,mycroft2020trees} for similar results in the directed setting.
    
    In this paper, we will look for \textit{pseudorandom properties} in a graph that force the containment of a subdivision of the complete graph which is also spanning. 
 	
	\subsection{$(n, d, \lambda)$-graphs}
	Informally speaking, we say that a graph $G$ is \emph{pseudorandom} if $G$ resembles certain properties that appear with high probability\footnote{We say that the binomial random graph $G(n,p)$ satisfies a property $\mathcal P$ with high probability (w.h.p.) if $\mathbb P(G(n,p)\in\mathcal P)=1-o(1)$ when $n$ goes to infinity.} in a random graph. A convenient way to define pseudorandomness in graphs is via a spectral gap condition, as follows. We say an $n$-vertex graph $G$ is an $(n,d,\lambda)$-graph if $G$ is $d$-regular and all the non-trivial eigenvalues of the adjacency matrix of $G$ are bounded (in absolute value) by $\lambda$. The well-known Expander Mixing Lemma states that the distribution of the edges of an $(n,d,\lambda)$-graph $G$ is controlled by $\lambda$, and thus the smaller that $\lambda$ is, the closer to random that $G$ looks like. We recommend the survey of Krivelevich and Sudakov~\cite{krivelevich2006pseudo} for a detailed exposition about pseudorandom graphs.
	
	One of the central problems in pseudorandom graphs is the following: How large does $d$ need to be (with respect to $\lambda$) so that every $(n,d,\lambda)$-graph contains a copy of certain spanning graph $F$? This problem has shown to be quite challenging, as the only cases that are fully understood are perfect matchings~\cite{krivelevich2006pseudo}, triangle-factors~\cite{morrisfactors}, and only recently the Hamiltonicity problem was settled by Dragani\'c, Montgomery, Munh\'a Correia, Pokrovskiy and Sudakov~\cite{draganic2024hamiltonicity} who solved a central problem in the area posed by Krivelevich and Sudakov~\cite{KrivelevichHamilton} in 2003. A natural next step after Hamiltonicity is then to look for spanning trees, in which case Alon, Krivelevich and Sudakov~\cite{alon2007embedding} conjectured that for every $\Delta\ge 2$, there is some constant $C$ so that if $G$ is an $(n,d,\lambda)$-graph with $d/\lambda\ge C$, then it contains all $n$-vertex trees with maximum degree at most $\Delta$. This conjecture is true for trees with many leaves, as shown by Pavez-Sign\'e~\cite{pavez2023spanning}, but in general is wide-open, as the current record is $d/\lambda=\Omega(\log^3 n)$ as proved by Hyde, Morrison, M\"uyesser and Pavez-Sign\'e~\cite{hyde2023spanning}. 
	
	Our first result is that $(n,d,\lambda)$-graphs with $d/\lambda\gg 1$ contain a spanning subdivision of the complete graph even for unbounded degrees.  
    
    \begin{theorem}\label{thm:pseudorandom:2}
        There exist constants $C>0$ and $0<c\leq1$ such that the following holds. If $G$ is an $(n,d,\lambda)$-graph with $d/\lambda\ge C$, then $G$ contains a spanning $K_t$-subdivision for all $2\le t\le \min\{cd,c\sqrt{\frac{n}{\log n}}\}$. 
    \end{theorem}
    
	A condition of the form $t\le cd$ is certainly needed, as each \textit{branch vertex} of a $K_t$-subdivision has degree $t-1$ and the host graph is $d$-regular. If we insist on the $K_t$-subdivision having each edge subdivided, then $t=O(\sqrt{n})$ is needed as each of the $\binom{t}{2}$ \textit{branch paths} adds at least one extra vertex. The log-factor in \Cref{thm:pseudorandom:2} is also needed, as there are celebrated constructions of $(n,d,\lambda)$-graphs with girth $\Omega(\log n)$ (see Example~11 in \cite{krivelevich2006pseudo}) showing that \Cref{thm:pseudorandom:2} is best-possible up-to the multiplicative constant.  
The only drawback in \Cref{thm:pseudorandom:2} is that one of the paths in the $K_t$-subdivision is notoriously longer than the rest of the branching paths. If one insists on all paths having roughly the same length, then a different approach is needed. To address this problem, we use the existence of \textit{sorting networks} in sparse expanders, showing that $d/\lambda\gg (\log n)^3$ suffices to guarantee the containment of a nearly-balanced subdivision, which is also spanning.  

\begin{theorem}\label{thm:pseudorandom:1}
    There exist constants $C>0$ and $0<c\leq1$ such that the following holds. If $G$ is an $(n,d,\lambda)$-graph with $d/\lambda\ge C\log^3n$, then $G$ contains a  spanning nearly-balanced $K_t$-subdivision for all $2\le t\le \min\{cd,c\sqrt{\frac{n}{\log^3n}}\}$.
\end{theorem}

The extra polylog-factor comes from the fact that the depth of the sorting network we are using is $\Theta(\log^3n)$, which forces us to use connecting paths of length $\Omega(\log^3 n)$.
We believe the right bound on $t$ should be of the form $t\ll \sqrt{\frac{n}{\log n}}$ as in \Cref{thm:pseudorandom:2}.

\begin{conjecture}
    There exist positive constants $C > 0$ and $0 < c < 1$ such that the following holds. If $G$ is an $(n,d,\lambda)$-graph with $d/\lambda\ge C$, then $G$ contains a spanning nearly-balanced $K_t$-subdivision for all $2\le t\le \min\{cd,c\sqrt{\frac{n}{\log n}}\}$.
\end{conjecture}

\subsection{Graphs without large bipartite holes}

Arguably, the weakest notion of pseudorandomness that has been considered in the literature is that of having no large bipartite holes, as in the following definition. 

\begin{definition}\label{def:joined}
    For $m\in\mathbb N$, say a graph $G$ is \emph{$m$-joined} if for every pair of disjoint subsets $A, B \subset V(G)$ with $|A|=|B|=m$, there is an edge between them.
\end{definition}
    
Observe that \Cref{def:joined} is a fairly weak pseudorandom property as, for instance the binomial random graph $G(n,p)$ is w.h.p. $m$-joined for $m=\lceil{5\log(np)}/{p}\rceil$ and $pn>20$ (see Proposition~3.26 in~\cite{montgomery2019spanning}). On the other hand,  if $G$ is an $(n,d,\lambda)$-graph, the expander mixing lemma implies that $G$ is $\frac{\lambda n}{d}$-joined for every choice of the parameters $n,d$ and $\lambda$. Moreover, if $G$ is an $m$-joined graph, then there is no guarantee of $G$ containing any particular spanning connected structure, as $G$ might contain isolated vertices. For example, the graph consisting of a copy of $K_{n-m+1}$ together with $m-1$ isolated vertices is $m$-joined but not connected. One can even add a perfect matching from the isolated vertices to the complete graph $K_{n-m+1}$ and still fail to find a Hamilton path. This suggests one needs some sort of extra condition, and a usual choice is to add a minimum degree condition on top of the $m$-joined property. The first result in this direction is due to McDiarmid and Yolov~\cite{mcdiarmid2017hamilton}, who showed that every $n$-vertex $o(n)$-joined graph with a linear minimum degree is Hamiltonian (they even provide the optimal dependency between the minimum degree and the size of the bipartite holes). Similar results have been proved for the containment of spanning trees with bounded degree~\cite{han2023spanning}, for an almost decomposition into almost spanning trees of bounded degree~\cite{kim2020tree}, and for loose spanning trees in hypergraphs~\cite{han2024ramsey}. 
 
	Our last result says that if $G$ has a linear minimum degree and no large bipartite holes, then we can find a spanning $K_t$-subdivision even for $t=O(\sqrt{n})$. 
	
	\begin{theorem}\label{thm:main}
		For every $\mu>0$, there exist constants $0 < c,\varepsilon \leq 1$ and $n_0\in\mathbb N$ such that the following holds for all $n\ge n_0$. If $G$ is an $n$-vertex $\varepsilon n$-joined graph with $\delta(G)\ge \mu n$, then $G$ contains a spanning nearly-balanced $K_t$-subdivision for all $2 \le t \le c\sqrt{n}$. 
	\end{theorem}
A nice implication of Theorem~\ref{thm:main} is a result regarding the \textit{randomly perturbed graph model}. In this model, one starts with an $n$-vertex  (deterministic) graph $G_\mu$ satisfying $\delta(G_\mu)\geq \mu n$, and then sprinkles a few random edges on top of $G_\mu$, say by using an instance of the binomial random graph $G(n,p)$. This model is interesting as it attempts to describe graphs which are essentially deterministic but for whose overall structure there may be some slight uncertainty. A desired result in this setting would be to sprinkle as few edges as possible whilst guaranteeing that $G_\mu \cup G(n,p)$ satisfies certain properties with high probability. For example, the above-mentioned result of McDiarmid and Yolov~\cite{mcdiarmid2017hamilton} implies a result of Bohman, Frieze and Martin~\cite{bohman2003many}, which says that $G_\mu\cup G(n,C/n)$ is w.h.p. Hamiltonian as long as $C=C(\mu)$ is large enough. As a corollary of Theorem~\ref{thm:main}, we obtain the following result for spanning clique subdivisions in the randomly perturbed setting. 

\begin{corollary}\label{cor:randomly-perturbed}
    For every $\mu>0$, there exist constants $C > 0$ and $0 < c \leq 1$ such that the following holds. Let $G$ be an $n$-vertex graph with $\delta(G)\ge \mu n$ and let $p\ge C/n$. Then w.h.p. $G\cup G(n,p)$ contains a spanning nearly-balanced $K_t$-subdivision for all $2\le t \le c\sqrt{n}$. 
\end{corollary}

\section{Preliminaries}\label{sec:prelim}
We use standard graph theory notation. For a graph $G$, we let $V(G)$ and $E(G)$ denote the set of vertices and edges of $G$, respectively, and write $|G|=|V(G)|$ for the size of $G$. Given a vertex $x\in V(G)$, we denote by $N(x)$ the set of neighbours of $x$ and write $d(x):=|N(x)|$ for the degree of $x$. We let $\delta(G)$ denote the minimum of $d(x)$ among all vertices $x\in V(G)$. Given a subset $U\subset V(G)$, we write $N(x,U)$ for the set of neighbours of $x$ in $U$ and write $d(x,U):=|N(x,U)|$ for the degree of $x$ in $U$. We write $\Gamma(U)=\bigcup_{u\in U}N(u)$ for the neighbourhood of $U$ and $N(U)=\Gamma(U)\setminus U$ for the external neighbourhood of $U$. Given a subset $A\subset V(G)$, we write $G[A]$ for the graph induced by $A$ and for $B\subset V(G)\setminus A$, we denote by $G[A,B]$ the bipartite graph induced by $A$ and $B$, which has bipartition classes $A$ and $B$ and edges $ab\in E(G)$ with $a\in A$, $b\in B$. We will use subscripts to indicate which graph we are working with if we are dealing with more than one graph at the same time. For a graph $G$ and a subset of vertices $U \subset V(G)$, we let $I(U)$ denote the empty graph on the vertex set $U$.
    
A path in $G$ is a sequence of distinct vertices $P=v_1\ldots v_{t+1}$ such that $v_iv_{i+1}\in E(G)$ for all $i\in [t]$, in which case we say $v_1$ and $v_{t+1}$ are the endpoints of $P$ and the interior of $P$ is the set $\text{int}(P)=\{v_2,\ldots, v_t\}$.  The size of a path $P$ is the number of vertices of $P$ and the length of $P$ is its number of edges. For vertices $a,b\in V(G)$, an $(a,b)$-path is a path with $a$ and $b$ as endpoints and for sets $A,B\subset V(G)$, an $(A,B)$-path is a path with one endpoint in $A$ and another in $B$. Given a subgraph $H\subset G$ and a path $P\subset G$, we let $H+P$ denote the graph with vertex set $V(H+P)=V(H)\cup V(P)$ and edge set $E(G+P)=E(G)\cup E(H)$.  If $F$ is a fixed graph and $F'$ is an $F$-subdivision, the set of branch vertices of $F'$ is the image of the vertices in $V(F)$ and the set of branching paths of $F'$ corresponds to the set of all paths connecting distinct branch vertices.  
 	
For positive integers $k\leq n \in \mathbb N$, we let $[n]=\{1,\ldots, n\}$, $[n]_0=[n]\cup \{0\}$ and we also let ${[n]}\choose{k}$ denote the set of $k$-element subsets of $[n]$. Given numbers $a,b\in\mathbb R$ and $c>0$, we write $a=b\pm c$ to denote that $b-c\le a\le b+c$. Also, we use standard hierarchy notation; that is, we write $a\ll b$ if given $b$ one can choose $a$ sufficiently small so that all the subsequent statements hold. We will omit floor and ceiling symbols in order to avoid cumbersome notation, as long as it does not affect the arguments. 
	
	
\subsection{Probabilistic tools}\label{subsec:probabilistic-tools}
We will need two probabilistic results stating that degrees of vertices are inherited by taking random samples in graphs. The first result can be proved using the local lemma (see Section~2.4.2 in~\cite{montgomery2024approximate} for a proof of a stronger result).
    
\begin{lemma}\label{lemma:randomsets2}
    Let $s\ge 1$ be fixed and let $1/d\ll \gamma\ll p\ll 1/s$ and let $p_1,\ldots, p_s\in (0,1)$ with $p_i\ge p$, for $i\in [s]$, and $p_1+\ldots+p_s=1$. Let $G$ be a $d$-regular graph and let $V(G)=X_1\cup \ldots \cup X_s$ be a random partition where each vertex $v\in V(G)$ belongs to $X_i$ with probability $p_i$, making all choices independently. Then, with positive probability, for each $i\in [s]$, $|X_i|=(1\pm\gamma)p_in$ and, for each $v\in V(G)$, $d(v,X_i)=(p_i\pm \gamma)d$.
\end{lemma}
In the second result, we partition the vertex set into $O(\log^4n)$ sets at the cost of needing a stronger bound on the minimum degree. To do so, we use the following result (see Lemma 3.1 in~\cite{hyde2023spanning}).

\begin{lemma}\label{lemma:randomsets} 
    Let $1/n\ll \gamma\ll 1$ and let $G$ be an $n$-vertex graph with $\delta(G)\geq \log^6 n$.
    Let $R$ be a uniformly random subset of $V(G)$ of size $k\geq n/\log^4 n$, and let $v\in V(G)$. Then, \[\mathbb{P}\left(d(v, R)=(1\pm \gamma)d(v)\cdot\tfrac{k}{n}\right)\geq 1-n^{-2}.\]
\end{lemma}

\begin{corollary}
    Let $1/n\ll\gamma\ll 1$ and let $G$ be an $n$-vertex graph with $\delta(G)\ge \log^6 n$. Let $s\ge 2$ and suppose $n$ decomposes as $n=n_1+\ldots +n_s$, where $n_i\ge n/\log^4 n$ for each $i\in [s]$. If $V(G)=V_1\cup\ldots\cup V_s$ is a partition chosen uniformly at random subject to $|V_i|=n_i$ for each $i\in [s]$, then, with probability $1-o(1)$, for every vertex $v\in V(G)$ and $i\in [s]$, $d(v,V_i)=(1\pm \gamma)d(v)\cdot \frac{n_i}{n}$.
\end{corollary}

\begin{proof} 
    Note that, for each $i\in [s]$, $V_i$ has the same distribution as a uniformly chosen $n_i$-element subset of $V(G)$. For each vertex $v\in V(G)$ and index $i\in [s]$, Lemma~\ref{lemma:randomsets} gives 
    \[\mathbb P\big(d(v,V_i)=(1\pm \gamma)d(v)\cdot \tfrac{n_i}{n}\big)\ge 1-n^{-2}.\] Therefore, by a union-bound over all vertices $v\in V(G)$ and indices $i\in [s]$, we have
    \[\mathbb P\left(\exists v\in V(G)~\text{and } i\in[s] \text{ s.t. }\big|d(v,V_i)-d(v)\cdot \tfrac{n_i}{n}\big|>\gamma d(v)\cdot \tfrac{n_i}{n}\right)\le  s\cdot n\cdot n^{-2}\le \frac{\log^4 n}{n}.\] 
\end{proof}


\subsection{Pseudorandom properties of $(n,d,\lambda)$-graphs}
The following lemma collects all the basic properties of $(n,d,\lambda)$-graphs that will be used throughout the paper (see~\cite{krivelevich2006pseudo} for the proofs). 
	
\begin{lemma}\label{lemma:pseudorandom}
    For an $(n,d,\lambda)$-graph $G$, the following properties hold.
    \begin{enumerate}
        \item \label{lemma:secondeig}$\lambda\ge \sqrt{d\cdot\frac{n-d}{n-1}}$.   
			
	\item \label{lemma:joined}$G$ is $\frac{\lambda n}{d}$-joined.
			
        \item\label{lemma:expander}For every pair of (not necessarily disjoint) sets $A,B\subset V(G)$, we have 
        \[\left|e(A,B)-\tfrac{d}{n}|A||B|\right|<\lambda\sqrt{|A||B|}.\]
    \end{enumerate}
\end{lemma}

The next lemma says that minimum degree conditions in $(n,d,\lambda)$-graphs are enough to force the expansion property of small sets.  

\begin{lemma}\label{lemma:expansion:1}
    Let $D, k\ge 1$ and let $d,\lambda>0$ satisfy $d/\lambda > k D$. Let $G$ be an $(n,d,\lambda)$-graph which contains a subset $X\subset V(G)$ such that every vertex $v\in V(G)$ has at least $k D\lambda$ neighbours in $X$. Then, every subset $S\subset V(G)$ of size $|S|\le (k-1)\lambda n/d$ satisfies $|N(S)\cap X|\ge (D - 1)|S|$.
\end{lemma}

\begin{proof}
    Let $S\subseteq V(G)$ be a vertex subset of $G$ with $|S| \leq (k-1)\lambda n / d$.
    By \Cref{lemma:pseudorandom}~\ref{lemma:expander}, the following inequality holds.
    \begin{equation}\label{eq:1}
        kD\lambda |S| \leq e(S, \Gamma_G(S)\cap X) \leq \frac{d}{n}|S||\Gamma(S)\cap X| + \lambda \sqrt{|S||\Gamma(S)\cap X|}.    
    \end{equation}
    Assume $|\Gamma(S)\cap X| < D|S|$. Then \eqref{eq:1} implies that we have
    \[ 
        kD\lambda |S| < \frac{d D}{n}|S|^2 + \lambda \sqrt{D} |S|.
    \]
    Thus, we deduce that $|S| > (k-1)\lambda n/ d$, a contradiction.
    Hence, we have $|N(S) \cap X| \geq |\Gamma(S)\cap X| - |S| \geq (D-1)|S|$. This completes the proof.
\end{proof}

The last result of this section gives a sufficient condition to find perfect matchings between small subsets of $(n,d,\lambda)$-graphs. 
	
\begin{lemma}[Lemma 3.6 in~\cite{hyde2023spanning}]\label{lemma:matching}
    Let $0 < 1/C \ll\varepsilon\ll 1$ and let $n,d\in\mathbb N$ such that $n\ge 3$ and $\lambda >0$ satisfy $d/\lambda \ge C\log^3n$. Suppose $G$ is an $(n,d,\lambda)$-graph that contains disjoint subsets $A,B\subset V(G)$ with $|A|=|B|$ and $\delta (G[A,B])\ge \varepsilon d/2\log^3n$. Then, $G[A,B]$ contains a perfect matching.    
\end{lemma}

	
\subsection{The extendability method}
The main tool we will use in the proof of \Cref{thm:pseudorandom:2} and \Cref{thm:pseudorandom:1} is the \emph{extendability method}, which was first introduced by Friedman and Pippenger~\cite{FP1987} in 1987 and subsequently developed over the years (see e.g.~\cite{draganic2022rolling, H2001,montgomery2019spanning, montgomery2023ramsey}). 
	
\begin{definition}\label{def:dmextendable}
    Let $D,m\in\mathbb N$ with $D\ge 3$. Let $G$ be a graph and let $S\subset G$ be a subgraph with $\Delta(S) \leq D$.
    We say that $S$ is \emph{$(D,m)$-extendable} if for all $U\subset V(G)$ with $1\le |U|\le 2m$ we have
		
    \begin{equation}\label{def:extendability}
        |\Gamma_G(U)\setminus V(S)|\ge (D-1)|U|-\sum_{u\in U\cap V(S)}(d_S(u)-1).
    \end{equation}
\end{definition}

The following result states that it is enough to control the external neighbourhood of small sets to verify extendability.

\begin{proposition}\label{prop:weakerextendability}
    Let $D,m\in\mathbb N$ with $D\ge 3$. Let $G$ be a graph and let $S\subset G$ be a subgraph with $\Delta(S)\le D$. 
    Assume for all $U\subset V(G)$ with $1\leq|U|\leq 2m$, the following inequality holds 
	\[|N_G(U)\setminus V(S)|\geq D|U|.\]
    Then $S$ is $(D,m)$-extendable in $G$.
\end{proposition}

The main property of the extendability method is that, given~an {extendable} subgraph $S$ and a vertex $x\in V(S)$ of small degree in $S$, one can find an edge $xy$, with $x\in V(S)$ and $y\not\in V(S)$, so that $S+xy$ remains extendable. This key property can be iterated to find large structures in expander graphs, such as long cycles or large trees.
In the proof of our main theorems, we frequently use the following lemma, which says that in an extendable subgraph $S$, one can connect any pair of distinct vertices from $S$ by paths of logarithmic length, as long as some mild conditions hold.

\begin{lemma}[Corollary 3.12 in {\cite{montgomery2019spanning}}]\label{lemma:connecting}
    Let $D,m\in\mathbb N$ with $D\ge 3$, and let $k=\lceil \log (2m)/\log (D-1)\rceil$.
    Let $\ell\in\mathbb N$ satisfy $\ell\ge 2k+1$ and let $G$ be an $m$-joined graph which contains a $(D,m)$-extendable subgraph $S$ of size $|S|\le |G|-10Dm-(\ell-2k-1)$.
    Suppose that $a$ and $b$ are two distinct vertices in $S$ with $d_S(a),d_S(b)\le D/2$. 
    Then, there exists an $a,b$-path $P$ of length $\ell$ such that
    \begin{enumerate}
        \item[$(i)$] all internal vertices of $P$ lie outside $S$, and
        \item[$(ii)$] $S+P$ is $(D,m)$-extendable. 
    \end{enumerate}

\end{lemma}

	
\section{Proofs}
\subsection{Proof of \Cref{thm:pseudorandom:2}}
We need a result establishing that expander graphs are \emph{Hamilton-connected}, which means that for every pair of distinct vertices, there is a Hamilton path with those vertices as endpoints. 

\begin{definition}
    We say that a graph $H$ is a $C$-\textit{expander} if $H$ is $\frac{n}{2C}$-joined and
    \stepcounter{propcounter}
    \begin{enumerate}[label =\textbf{\Alph{propcounter}}]
        \item\label{def:expander:1} for all sets $X\subset V(G)$ with $1\le |X|\le \frac{n}{2C}$, $|N(X)|\ge C|X|$.
    \end{enumerate}
\end{definition}
   
\begin{theorem}[{\cite[Theorem 7.1]{draganic2024hamiltonicity}}]\label{lemma:hamilton connected}
    There exists a universal constant $C_0 > 0$ such that the following holds for all $C \geq C_0$. If a graph $G$ is a $C$-expander, then $G$ is Hamilton-connected. 
\end{theorem}

\begin{proof}[Proof of \Cref{thm:pseudorandom:2}]
    We introduce auxiliary constants satisfying \[0 < 1/C\ll 1/d\ll\gamma\ll \varepsilon\ll c, p\ll 1\]
    so that $C\ge 100C_{\ref{lemma:hamilton connected}}/p$. Let $V(G)=X_1\cup X_2\cup X_3$ be a random partition so that for each $v\in V(G)$, $\mathbb P(v\in X_2)=\mathbb P(v\in X_3)=p$ and $\mathbb P(v\in X_1)=1-2p$, and all choices are made independently. Then, \Cref{lemma:randomsets2} implies that there exist pairwise disjoint sets $X_1,X_2,X_3$ such that $|X_1|=(1-2p\pm\gamma)n$, $|X_2|,|X_3|=(p\pm\gamma)n$ and for each $v\in V(G)$,
    
    \begin{enumerate}[label =({$\star$})]
        \item\label{thm:pseudorandom:2:degrees} $d(v,X_1)=(1-2p \pm \gamma)d$, $d(v,X_2)=(p\pm \gamma)d$, and $d(v,X_3)=(p\pm \gamma)d$.
    \end{enumerate}
    Our first step is to find the branch vertices of the $K_t$-subdivision.
 
    \begin{claim}\label{claim1:thm:pseudorandom:2} 
        There is a collection $\{S_1,\ldots, S_t\}$ of vertex-disjoint stars with $t-1$ leaves, where for each $i\in [t]$, the star $S_i$ is in $G[X_1]$.
    \end{claim}
    
    \begin{claimproof}
        Let $\{S_1,\ldots, S_\ell\}$ be a maximal collection of vertex-disjoint stars with $t-1$ leaves in $G[X_1]$. Suppose $\ell < t$ and let $Y = X_1\setminus \cup_{i\in [\ell]}V(S_i)$. Using that $|Y|\ge |X_1|-t^2\ge (1-4p)n$, the expander mixing lemma (Lemma~\ref{lemma:pseudorandom}~\ref{lemma:expander}) gives
        \[e(G[Y])\ge  \frac{1}{2} \left(\frac{d}{n}|Y|^2-\lambda |Y|\right)\ge \frac{|Y|}{2}\cdot \left( (1-4p)d-\lambda \right) \ge \frac{|Y|}{2}\cdot (1-5p)d,\]
        which implies that there is a vertex in $Y$ with at least $(1-5p)d\ge t$ neighbours in $Y$, contradicting the maximality of $S_1,\ldots, S_\ell$. The factor of $1/2$ in the first inequality comes from the observation that $e(Y,Y)=2 \cdot e(G[Y]).$
    \end{claimproof}

    For each $i\in [t]$, let $s_i\in V(S_i)$ be the centre of $S_i$ which we identify as the $i$-th branch vertex of the $K_t$-subdivision and let $L_i = V(S_i) \setminus \{s_i\}$ be the leaves of $S_i$. Set $Z=(X_1\cup X_2)\setminus \{s_1,\ldots, s_t\}$ and let $L: = I \left( \cup_{i\in [t]}L_i\right) \subset G[Z]$ be the empty graph. We now claim that $L$ is $(10,\frac{\lambda n}{d})$-extendable in $G[Z]$. 
    To prove this, we will show that each set $U\subset Z$ with $1\le |U|\le 2 \lambda n/d$ satisfies $|N(U)\cap (X_1\cup X_2)\setminus \cup_{i\in [t]}V(S_i)|\ge 10|U|$, from which the $(10, \lambda n/d)$-extendability is a consequence of \Cref{prop:weakerextendability}. Note that~\ref{thm:pseudorandom:2:degrees} implies that every vertex $v\in V(G)$ satisfies
    \[d_G(v,X_2)\ge (p-\gamma)d\ge \frac{pd}{2}\ge 40\lambda\]
    and thus, for any subset $U\subset Z$ of size $1\le |U|\le 2\lambda n/d$, \Cref{lemma:expansion:1} implies 
    \[|(N(U)\cap Z)\setminus V(L)|\ge |N(U)\cap X_2| \ge 10|U|.\]
    Our goal now is to find $\binom{t}{2}-1$ vertex-disjoint paths joining distinct branch vertices so that the size of the unused vertices in $X_1\cup X_2$ is at most $200\varepsilon n$. Let $\mathcal Q$ be the set of pairs $(i,j)$ with $1\le i<j\le t$ such that $(i,j)\not =(1,t)$, and find a collection of positive integers $(\ell_{e})_{e\in \mathcal Q}$  such that 

    \begin{enumerate}
        \item\label{thm2:length} $\ell_{e}\ge  100\log n$ for each $e\in \mathcal Q$, and 
        \item\label{thm2:leftover:size} $(1-p)n-200\varepsilon n \leq \sum_{e\in \mathcal Q}\ell_e\le (1-p)n-100\varepsilon n.$
    \end{enumerate} 
This is clearly possible as $t^2\le c^2n/\log n$ and $c$ is small enough.
    \begin{claim}
        There is a collection of vertex-disjoint paths $(P_e)_{e\in\mathcal Q}$ with internal vertices in $Z$ such that, for $e=(i,j)$, $P_e$ is an $(s_i,s_j)$-path of length $\ell_e$.
    \end{claim}

    \begin{claimproof}
        We will construct a collection of paths $(P'_e)_{e\in \mathcal Q}$ connecting distinct vertices in $L$ and using only vertices from $Z\setminus V(L)$ in the process. Say a vertex $v\in V(L)$ is saturated if it has already been used in one of the constructed paths. For each $e=(i,j)\in \mathcal Q$ in turn, find a path $P'_{e}$ of length $\ell_e-2$ in $G[Z]$ so that $P'_e$ is an $(L_i,L_j)$-path connecting two unsaturated vertices and such that $L$ together with all the paths we have constructed so far is a $(10,\frac{\lambda n}{d})$-extendable subgraph of $G[Z]$. If this algorithm ends with $\binom{t}{2}-1$ paths $(P'_e)_{e\in\mathcal Q}$ as described above, we can extend each of these paths to their respective branch vertex to find the desired collection of paths. So suppose the algorithm stops with a collection of paths $P'_{e_1},\ldots, P'_{e_r}$ so that $r<\binom{t}{2}-1$ and $L+P'_{e_1}+\ldots +P'_{e_r}$ is $(10,\frac{\lambda n}{d})$-extendable. Let $e=(i,j)\in\mathcal Q\setminus\{e_1,\ldots, e_r\}$. Note that both $L_i$ and $L_j$ contain at least one unsaturated vertex, say $v_i\in L_i$ and $v_j\in L_j$, as the path $P'_e$ is missing. As $\ell_e-2\ge 100\log n-2\ge 10\log_9(\frac{2\lambda n}{d})$ and 
        \[|L+P'_{e_1}+\ldots +P'_{e_r}|+(\ell_e-2)\le \sum_{e\in\mathcal Q}\ell_e\le (1-p)n-100\varepsilon n\le |Z|-100\cdot \frac{\lambda n}{d},\]
        we can use \Cref{lemma:connecting} to find a $(v_i,v_j)$-path of length $\ell_e-2$ in $G[Z]$ whose internal vertices avoid $V(L+P'_{e_1}+\ldots+P'_{e_r})$ and such that $(L+P'_{e_1}+\ldots +P'_{e_r})+P'_e$ is $(10,\frac{\lambda n}{d})$-extendable in $G[Z]$. This proves the algorithm ends with all the required paths, finishing the proof. 
    \end{claimproof}
    
    Note that up to this point, we have found all but one of the paths in the $K_t$-subdivision. Thus, it only remains to find a $(s_1,s_t)$ path using all the remaining vertices in the graph at the same time. Let $Z'=(Z\cup X_3)\setminus \cup_{e\in\mathcal Q}V(P'_e)$ and let $u_1\in L_1\cap Z'$ and $u_t\in L_t\cap Z'$ be the remaining leaves after $\cup_{e\in\mathcal Q}P_e$ is constructed. We know $d/\lambda \geq C\geq 100C_{\ref{lemma:hamilton connected}}/p$, so suppose $d/\lambda=100C'/p$ for some $C' \geq C_{\ref{lemma:hamilton connected}}$.
     \begin{claim}
        $G[Z']$ is a $C'$-expander.
    \end{claim}

    \begin{claimproof}
        Firstly, note that $|Z'|\leq (p+\gamma)n+200\varepsilon n \leq1.02pn$. Thus, we just need to check that $G[Z']$ satisfies~\ref{def:expander:1}, as $G[Z']$ is $\frac{\lambda n}{d}$-joined (Lemma~\ref{lemma:pseudorandom}~\ref{lemma:joined}) and $\frac{\lambda n}{d}=\frac{pn}{100C'}\leq \frac{|Z'|}{2C'}$.
         To check~\ref{def:expander:1}, we need to verify the expansion property of subsets $U \subset Z'$ of size at most $|Z'|/2C'\leq0.51pn/C'$. Observe that~\ref{thm:pseudorandom:2:degrees} implies each vertex $v\in Z'$ satisfies $d(v,Z')\ge d(v,X_3)\ge 0.99pd=99C'\lambda$.
         
        An application of Lemma~\ref{lemma:expansion:1} with $k_{\ref{lemma:hamilton connected}}=90$ and $D_{\ref{lemma:hamilton connected}}=1.1C'$ implies that sets $U \subset Z'$ of size $|U| \leq \frac{(k-1)\lambda n}{d}=\frac{89pn}{100C'}$ expand by a factor of at least $(1.1C'-1)\geq C'$, thus proving the claim.
    \end{claimproof}

    Finally, use Theorem~\ref{lemma:hamilton connected} to find a Hamilton path in $G[Z']$ with endpoints $v_i$ and $v_j$ to complete the spanning $K_t$-subdivision. 
\end{proof}


\subsection{Proof of Theorem~\ref{thm:pseudorandom:1}}
In this section, we find a $K_t$-subdivision where all branching paths have nearly the same length. To do so, we rely on the existence of sorting networks in sparse expanders as proved by Hyde, Morrison, M\"uyesser and Pavez-Sign\'e~\cite{hyde2023spanning}.
	
	\begin{lemma}[Lemma 2.3 in~\cite{hyde2023spanning}]\label{lemma:find_sorting_network_in_expander} There is a constant $C$ such that the following holds. Let $1/n\ll 1/K\ll 1/C$, and let $D,m\in \mathbb{N}$ satisfy $m\leq n/100D$ and $D\ge 100$. Suppose $G$ is an $n$-vertex $m$-joined graph which contains disjoint subsets $V_1, V_2\subseteq V(G)$ with $|V_1|=|V_2|\leq n/K\log^{3}n$ such that $I(V_1\cup V_2)$ is $(D,m)$-extendable in $G$.
		\par Then, for $\ell:=\lfloor C \log^3 n \rfloor$, there exists a $(D,m)$-extendable subgraph $S_{res}\subseteq G$ such that for any bijection $\sigma\colon V_1\to V_2$, there exists a $P_\ell$-factor of $S_{res}$ where each copy of $P_\ell$ has as its endpoints some $v\in V_1$ and $\sigma(v)\in V_2$.
	\end{lemma}

	Note that if $G$ is an $(n,d,\lambda)$-graph with $d\geq(1-\varepsilon)n$ for some $\varepsilon>0$ sufficiently small, then by Theorem~\ref{thm:main} we can already find a spanning nearly-balanced $K_t$-subdivision for all $2 \leq t \leq c\sqrt{n}$. Thus, when proving Theorem~\ref{thm:pseudorandom:1} we may assume that $d \leq (1-\varepsilon)n$, in which case it is well-known that $\lambda=\Omega(\sqrt{d})$~\cite{krivelevich2006pseudo}. This, combined with the statement of the theorem implies that $d\geq \log^6n$, provided $C$ is large enough. This allows us to use Lemma~\ref{lemma:randomsets}.
	\begin{proof}[Proof of Theorem~\ref{thm:pseudorandom:1}]We start by choosing constants 
		\[1/n_0\ll 1/C\ll\varepsilon\ll c\ll  1/K_{\ref{lemma:find_sorting_network_in_expander}}\ll 1/C_{\ref{lemma:find_sorting_network_in_expander}}\ll \alpha, 1/D\ll 1.\]
		Suppose that $n\ge n_0$ and let $G$ be an $(n,d,\lambda)$-graph with $d/\lambda\ge C\log^3 n$. Let $t\le \min\{cd,c\sqrt{\frac{n}{\log^3n}}\}$ be fixed and let $\ell=\lfloor C_{\ref{lemma:find_sorting_network_in_expander}}\log^3 n\rfloor$. Set $m=2c^2n/\log^3n$, $m'=m-\binom{t}{2}$ and $k=\lfloor(\log^3n)/(20c)\rfloor-1$, and pick disjoint random subsets $V_0,V_1,\ldots, V_{k+3}\subset V(G)$ subject to $|V_{0}|=\alpha n$, $|V_1|=\ldots=|V_k|=\varepsilon m$, and $|V_{k+1}|=|V_{k+2}|=|V_{k+3}|=m$. Then, by Lemma~\ref{lemma:randomsets}, we have w.h.p. that 
	
		\begin{enumerate}[label = $\bullet$]
			\item for every $v\in V(G)$ and $i\in [k+3]_0$, $d(v,V_i)=(1\pm 0.1)d\cdot \frac{|V_i|}{n}$.\end{enumerate}
Use Lemma~\ref{lemma:find_sorting_network_in_expander} to find a $(D,m)$-extendable subgraph $S_{res}\subset G-\cup_{i\in [k+1]_0}V_i$ of size $|S_{res}|=(\ell+1)\cdot m$ and such that $V_{k+2},V_{k+3}\subset V(S_{res})$ and 
	\stepcounter{propcounter}
	\begin{enumerate}[label =\textbf{\Alph{propcounter}}]
	\item\label{thm2:connecting} for every bijection $\sigma:V_{k+2}\to V_{k+3}$ there is a $P_\ell$-factor in $S_{res}$ such that each copy of $P_\ell$ has as its endpoints some $v\in V_{k+2}$ and $\sigma(v)\in V_{k+3}$.
	\end{enumerate}
Let $W=V(S_{res})\setminus (V_{k+2}\cup V_{k+3})$. The next claim finds a collection of vertex-disjoint stars, each of size $t$, and we omit its proof as it is identical to the proof of \Cref{claim1:thm:pseudorandom:2} in \Cref{thm:pseudorandom:2}. 

\begin{claim}
    There is a collection $\{S_1,\ldots, S_t\}$ of vertex-disjoint stars with $t-1$ leaves such that $\cup_{i\in [t]}S_i$ is disjoint from $\cup_{i\in [k+3]_0}V_i$ and $W$. 
\end{claim}
		
		For each $i\in [t]$, let $s_i$ be the centre of $S_i$ and let $L_i=\{u_{i,j}:j\in[t]\setminus\{i\}\}$ be the set of leaves of $S_i$. We will use $s_1,\ldots, s_t$ as the branch vertices of the $K_t$-subdivision. Let 
 $V'=\cup_{i\in [k]}V_i\cup W\cup\{s_1,\ldots, s_t\}$, $Z=\cup_{i\in [t]}L_i \cup V_{k+1}\cup V_{k+2}\cup V_{k+3}$, and set $G'=G-V'$. 
 
\begin{claim}
    $I(Z)$ is $(D,\frac{\lambda n}{d})$-extendable in $G'$.
\end{claim}

\begin{claimproof}
    By Proposition~\ref{prop:weakerextendability}, it suffices to show that for any $U \subset V(G')$ of size $1 \leq |U| \leq \frac{2\lambda n}{d}$, it holds true that $|N_{G'}(U)\setminus Z|\geq D|U|$. Recall that for every vertex $v \in V(G')$ we have that $d_G(v,V_0)\geq 0.9 \alpha d$. An application of Lemma~\ref{lemma:expansion:1} with $D_{\ref{lemma:expansion:1}}=D+1$ and $k_{\ref{lemma:expansion:1}}=3$ completes the proof as then for any $U \subset V(G')$ with $|U| \leq \frac{2\lambda n}{d}$, we have that $|N_{G'}(U)\setminus Z|\ge |N_G(U)\cap V_0|\geq D |U|.$
    
    
\end{claimproof}

We now briefly describe how we will connect the branch vertices $s_i$ and $s_j$. Clearly, it suffices to do so for $u_{i,j} \in L_i$ and $u_{j,i}\in L_j$. We only give an overview at this point, and postpone exact definitions until later. 
\newline

\begin{tikzpicture}
    \node (A) at (0,0) {$u_{i,j}$};
    \node (B) at (1.8,0) {$V_{k+1}$};
    \node (C) at (3.6,0) {$V'_1$};
    \node (D) at (5.4,0) {$V'_2$};
    \node (E) at (7.2,0) {$\cdots$};
    \node (F) at (9,0) {$V'_k$};
    \node (G) at (10.8,0) {$V_{k+2}$};
    \node (H) at (12.6,0) {$V_{k+3}$};
    \node (J) at (14.4,0) {$u_{j,i}$};

    \draw [->] (A) -- (B) node[midway, above] {$q_{i,j}$};
    \draw [->] (B) -- (C) node[midway, above] {$1$};
    \draw [->] (C) -- (D) node[midway, above] {$1$};
    \draw [->] (D) -- (E) node[midway, above] {$1$};
    \draw [->] (E) -- (F) node[midway, above] {$1$};
    \draw [->] (F) -- (G) node[midway, above] {$1$};
    \draw [->] (G) -- (H) node[midway, above] {$\ell$};
    \draw [->] (H) -- (J) node[midway, above] {$q_{j,i}$};
    \draw [->] (H) .. controls (7.2,-2) .. (B) node[midway, above] {$100\log n$};
    \node []  at (7.2,-2.1) {repeat $|J_{\{i,j\}}|$ times};
\end{tikzpicture}

The numbers above the arrows indicate the lengths of the corresponding path segments. The sets $V'_i$,  $i \in \left[ k \right]$, will satisfy $V_i \subset V'_i$ and $|V'_i|=m$, and the path lengths $q_{i,j}$, $q_{j,i} \in \mathbb{N}$ will be appropriately chosen as to ensure that the final spanning subdivision is nearly-balanced. The backward arrow indicates that we may visit some vertex parts multiple times (if $|J_{\{i,j\}}|$ is non-zero).

We let $\mathbb{M}=100\log n +\ell + k+1$, and note that this is the length of a single \textit{loop segment} and that $\mathbb{M} \leq \log^3n/(10c)$. Furthermore, observe that if the above sketch is true, then the length of the branching path from $u_{i,j}$ to $u_{j,i}$ will be equal to $q_{i,j}+q_{j,i}+ |J_{\{i,j\}}| \cdot \mathbb{M}+\ell +k +1$.

We now connect the vertices from $\cup_{i\in [t]}L_i$ with $V_{k+1}$ and $V_{k+3}$. To do so, we first pick subsets $A\subset V_{k+1}$ and $B\subset V_{k+3}$ of size $|A|=|B|=m':=m-\binom{t}{2}$. Letting $a_1,\ldots, a_{m'}$ and $b_1,\ldots, b_{m'}$ be a labelling of $A$ and $B$, respectively, use Lemma~\ref{lemma:connecting} to find a collection of paths $P_1,\ldots, P_{m'}$ such that, for each $i\in [m']$,

\begin{itemize}
	\item $P_i$ is an $(a_i,b_i)$-path of length $100\log n$, 
    \item $\text{int}(P_i)$ is disjoint from $I(Z)+P_1+\ldots+P_{i-1}$, and
	\item $I(Z)+P_1+\ldots +P_{i}$ is $(D,m)$-extendable.
\end{itemize}
This is clearly possible, as $m'\le m\le 2c^2n/\log^3n$ implies 

\begin{equation}\label{thm2:equation:size:Z}
    |I(Z)+P_1+\ldots +P_{m'}|\le t(t-1)+3m+m'\cdot 100\log n\le 10c^2n.
\end{equation}

Let $[m']=\cup_{\{i, j\}\in \binom{[t]}{2}} J_{\{i,j\}}$ be a partition so that $||J_{\{i,j\}}|-|J_{\{i',j'\}}||\le 1$ for all $\{i,j\}, \{i',j'\}\in \binom{[t]}{2}$. Pick a collection of positive integers $(q_{i,j})_{i\neq j\in[t]}$ such that for all $i\neq j\in [t]$ and $a\neq b \in [t]$,

\stepcounter{propcounter}
	\begin{enumerate}[label = \textbf{\Alph{propcounter}\arabic{enumi}}]
 \item\label{thm2:length:paths} $100\log^3 n\le q_{i,j},q_{j,i}\le 0.99n$,
 \item \label{thm2:samelength} $\Big|\left(q_{i,j} +q_{j,i}+|J_{\{i,j\}}|\cdot \mathbb{M}\right)-\left(q_{a,b}+q_{b,a}+|J_{\{a, b\}}| \cdot \mathbb{M}\right)\Big|\le 1$, and
    \item\label{thm2:divisibility} $n-|V'|-|Z|-m'(100\log n-1)- \sum_{1\leq i<j \leq t}\left((q_{i,j}-1)+(q_{j,i}-1)\right)=(1-\varepsilon)km$.
\end{enumerate}

This is possible as the number of leftover vertices after the $m'$ paths of length $100\log n$ are embedded is more than $n-|V'|-|Z|-100m'\log n\ge 0.8n$ and we want to find a collection of $t(t-1)\le c^2n/\log^3 n$ numbers, each of them greater than $100\log^3 n$.
Moreover, note that only the sum $(q_{i,j}+q_{j,i})$ matters, as one could increase either at the cost of decreasing the other and the final branching path's length will remain the same. Also, for each $i\neq j\in[t]$, $|J_{\{i,j\}}|=x~~ \textrm{or} ~~x+1$ for some number $x$ and this follows from construction. Hence, we must have that, for each $i\neq j \in [t]$, the number $(q_{i,j}+q_{j,i})$ contained in either the set $\{y, y+1\}$ or $\{y + \mathbb{M}, y + \mathbb{M} + 1\}$ for some number $y$, in order to satisfy~\ref{thm2:samelength}. From then on, we may `increase' $y$ so that condition~\ref{thm2:divisibility} is also satisfied.

Set $\mathcal Q=\emptyset $ and $F=I(Z)+P_1+\ldots +P_{m'}$, and iteratively, while $\mathcal Q\not=\binom{[t]}{2}$, do the following. Choose a pair $\{i, j\} \in\binom{[t]}{2}\setminus\mathcal Q$ and pick some unused vertices $u'_{i,j}\in V_{k+1}\setminus A$ and $u'_{j,i}\in V_{k+3}\setminus B$, and, if possible, perform the following step. 
\stepcounter{propcounter}
\begin{enumerate}[label =\textbf{\Alph{propcounter}$_{ij}$}]
    \item\label{thm2:step:paths} Find vertex-disjoint paths $Q_{j,i}$ and $Q_{i,j}$ such that $Q_{j,i}$ is a $(u_{j,i},u'_{j,i})$-path of length $q_{j,i}$, $Q_{i,j}$ is a $(u_{i,j},u'_{i,j})$-path of length $q_{i,j}$, $Q_{i,j}$ and $Q_{j,i}$ are internally disjoint from $F$, and $F+Q_{i,j}+ Q_{j,i}$ is $(D,\frac{\lambda n}{d})$-extendable in $G'$. Then we update the set $Q$ and the graph $F$ as $\mathcal Q:=\mathcal Q\cup\{ij\}$ and $F:=F+ Q_{i,j}+ Q_{j,i}$.
\end{enumerate}
We now show that it will be always possible to perform step~\ref{thm2:step:paths}. Indeed, by~\ref{thm2:divisibility}, we know that there are always at least $(1-\varepsilon)km\geq cn/20$ many unused vertices in $G'$. Thus, by Lemma~\ref{lemma:connecting} we may find the desired paths $Q_{i,j}$ and $Q_{j,i}$ to complete step~\ref{thm2:step:paths}.

Note that to finish the embedding of the $K_t$-subdivision, we only need to connect, for each $\{i,j\}\in \binom{[t]}{2}$, the vertices $u'_{ij}\in V_{k+1}$ and $u'_{ji}\in V_{k+3}$. Let $U$ denote the set of vertices that we have not yet used in $G'$. By~\ref{thm2:divisibility}, $|U|=(1-\varepsilon)km$ and so we may distribute these left-over vertices to the sets $V_1,V_2,\ldots,V_k$, to form pairwise disjoint sets $V_1',V_2',\ldots,V_k'$ satisfying
\begin{itemize}
    \item $|V'_1|=\ldots=|V'_k|=m$, 
    \item $V_i\subset V'_i$ for each $i\in [k]$, 
    \item $\left(V'_1\cup\ldots\cup V_k'\right)\setminus\left(V_1,\ldots,V_k\right)=V(G')\setminus V(F)$, and
    \item for every vertex $v\in V(G)$, $d(v,V'_i)\ge 0.9d\cdot \frac{|V_i|}{n}\ge \frac{\varepsilon c^2d}{2\log^3 n}$.
\end{itemize}
Use Lemma~\ref{lemma:matching}, iteratively, to find a collection of matchings $M_0, M_1,\ldots, M_{k+1}$ such that 
\begin{itemize}
    \item $M_0$ is a perfect matching in $G[V_{k+1},V'_1]$,
    \item for $i\in [k-1]$,  $M_i$ is a perfect matching in $G[V'_i,V'_{i+1}]$, and 
    \item $M_{k+1}$ is a perfect matching in $G[V'_k,V_{k+2}]$. 
\end{itemize}
The union of these $k+2$ matchings gives a family of $m$ vertex-disjoint paths of length $k+1$ so that each such path has as endpoints some vertex from $V_{k+1}$ and some vertex from $V_{k+2}$. It is thus left only to connect the endpoints of those paths in a suitable order to finish the embedding. 

To do so, we first partition $B$ into sets $B=\cup_{(i,j)\in \binom{[t]}{2}}B_{(i,j)}$, where $B_{(i,j)}=\{b^{(i,j)}_1,\ldots,b^{(i,j)}_{|J_{\{i,j\}}|}\}$. For each $(i,j)\in \binom{[t]}{2}$, let $v'_{i,j}\in V_{k+2}$ be the other endpoint of the path starting at $u'_{i,j}\in V_{k+1}$ and using edges from $M_0,\ldots, M_{k+1}$. For each $(i,j)\in \binom{[t]}{2}$, and $1\le z\le |J_{\{i,j\}}|$, let $c^{(i,j)}_z\in V_{k+2}$ be the other endpoint of the path which starts at $b^{(i,j)}_z$, goes through the corresponding path $P_l$, $l\in [m']$, and then uses edges from the matchings $M_0, M_1,\ldots, M_{k+1}$. Let $\sigma:V_{k+2}\to V_{k+3}$ be the permutation defined as follows: for each $(i,j)\in \binom{[t]}{2}$, set $\sigma(v'_{i,j})=b^{(i,j)}_1$, $\sigma(c_{z}^{(i,j)})=b^{(i,j)}_{z+1}$ for each  $1\le z<|J_{\{i,j\}}|$, and $\sigma(c_{|J_{\{i,j\}}|}^{(i,j)})=u'_{j,i}$. Then, using property~\ref{thm2:connecting}, we can complete the embedding of the $K_t$-subdivision.
\end{proof}


\subsection{Proof of Theorem~\ref{thm:main}}
The proof of Theorem~\ref{thm:main} uses the \textit{absorption method} as implemented by Montgomery~\cite{montgomery2019spanning} in his solution of Kahn's conjecture on the threshold of bounded degree spanning trees in random graphs.
	
\begin{definition}
    Given a graph $G$ and a vertex $v\in V(G)$, an absorber for $v$ is a tuple $(A,x,y)$ that consists of a subset $A\subset V(G-v)$ and vertices $x,y\in A$ such that 
    \stepcounter{propcounter}
	\begin{enumerate}[label=\textbf{\Alph{propcounter}}]
		\item\label{absorber:property1} both $G[A]$ and 
		      $G[A\cup \{v\}]$ contain a Hamiltonian $(x,y)$-path.
	\end{enumerate}
    The $(x,y)$-path $P_{x,y}$ with $V(P_{x,y})=A$ is called an \textit{absorbing path} for $v$. 
\end{definition}
	
Suppose that $(A,x,y)$ and $(A',x',y')$ are disjoint absorbers for vertices $v$ and $v'$, respectively. The key observation here is that if $P$ is an $(y,x')$-path which avoids $A\cup A'\cup\{v,v'\}$, then $(A\cup A'\cup V(P),x,y')$ is an absorber for both $v$ and $v'$. The main idea introduced by Montgomery~\cite{montgomery2019spanning} was that of using a sparse \textit{template} for combining absorbers of different vertices so that every single absorber is capable of `robustly' absorbing a `large' set of vertices. The following key lemma gives the existence of the template. 

\begin{lemma}[{\cite[Lemma 10.7]{montgomery2019spanning}}]\label{lemma:template}
    There is a constant $n_0$ such that, for every $n\geq n_0$ and $k\leq n$, there exists a bipartite graph H with $\Delta(H) \leq 100$ and bipartition classes $X$ and $Y \cup Z$, with $|X|=3n$ and $|Y|=2n$ and $|Z|=n+k$, such that the following holds. For every $Z' \subseteq Z$ with size $|Z'|=n$, there exists a perfect matching between $X$ and $Y \cup Z'$.
\end{lemma}
	
The last ingredient we need is that $m$-joined graphs contain almost spanning paths, a result which is obtained by a simple \textit{depth-first search} analysis of the graph.
	
\begin{lemma}\label{lemma:long-path}
    If $G$ is an $m$-joined graph, then $G$ contains a path of length at least $|G|-2m$.
\end{lemma}
	
We are now ready for the proof.

\begin{proof}[Proof of Theorem~\ref{thm:main}]	
    We start by choosing constants 
    \[0 < 1/n_0\ll c\ll\varepsilon  \ll \theta \ll \alpha\ll \mu < 1.\] 

    Suppose that $n\ge n_0$ and that $G$ is an $n$-vertex graph which is $\varepsilon n$-joined and has minimum degree $\delta(G)\geq \mu n$.  During the proof, we will repeatedly use the following observation. Let $U, Z\subset V(G)$ and let $v,v'$ be distinct vertices.

    \stepcounter{propcounter}
	\begin{enumerate}[label = \textbf{\Alph{propcounter}}]
            \item\label{property:connecting} If $d(v,U),d(v',U)>|Z|+4\varepsilon n$, then there is a $v,v'$-path of length 3 whose interior vertices belong to $U\setminus Z$.
	\end{enumerate}
    
    Indeed, if $d(v,U),d(v',U)>|Z|+4\varepsilon n$,  we can find disjoint sets $X\subset N(v,U)\setminus(Z\cup\{v'\})$ and $Y\subset N(v',U)\setminus(Z\cup\{v\})$ such that  $|X|=|Y|=\varepsilon n$, and then find an edge between $X$ and $Y$ using that $G$ is $\varepsilon n$-joined.\\
		
    \noindent\textit{Step 1. Fixing the reservoirs.} 
    Let $1\le t\le c\sqrt{n}$ be fixed from now on, and let $r=\theta n$ and $k=4\cdot \binom{t}{2}+8\varepsilon n$. We start by setting aside three disjoint random sets $R_1, R_2, R_3\subset V(G)$ with the following properties.
		
    \stepcounter{propcounter}
	\begin{enumerate}[label = \textbf{\Alph{propcounter}\arabic{enumi}}]
            \item\label{thm1:size} $|R_1|=2r$, $|R_2|=r +k$, and $|R_3|=\alpha n$.
            \item\label{thm1:degree:reservoir}  For all $v\in V(G)$ and $i\in[3]$, $d(v,R_i)\ge\frac{\mu |R_i|}{2}$. 
            \item\label{thm1:degree:graph} For all $v\in V(G)$, $d(v,V(G)\setminus \cup_{i\in [3]}R_i)\ge \frac{\mu n}{2}$.
	\end{enumerate}
  
    \noindent\textit{Step 2. Constructing the absorbers.} 
    Our aim now is to construct a collection of absorbing paths that is capable of absorbing $R_1$ together with any set of $r$ vertices from $R_2$. To do so, we first construct a collection of 100 vertex-disjoint absorbers for each vertex in $R_1\cup R_2$. 
		
    \begin{claim}\label{clm:disjoint-absorbers}
        For each $v\in R_1 \cup R_2$ there exists a collection $\{(\{a_{v,i},b_{v,i}\},a_{v,i},b_{v,i})\}_{i\in [100]}$ of absorbers for $v$ such that $a_{v,i},b_{v,i}\in V(G)\setminus \cup_{j\in [3]}R_j$ and all the absorbers are pairwise disjoint.
    \end{claim}
	
    \begin{claimproof}[Proof of \Cref{clm:disjoint-absorbers}] 
        We observe first that an absorber of size 2 for $v$ is just a triangle supported at $v$. We will find the triangles greedily using~\ref{thm1:degree:graph} and the joinedness condition. Indeed, firstly note that for every vertex $v\in V(G)$ we have 
    
        \begin{equation}\label{eq:triangle}
            d(v,V(G)\setminus\cup_{i\in [3]}R_i)\overset{\ref{thm1:degree:graph}}{\ge}\frac{\mu n}{2}\ge \frac{\mu n}{4}+200|R_1\cup R_2|,
        \end{equation}
        as $1/n_0\ll \alpha,\theta\ll \mu$.
			
        Secondly, for each vertex $v\in R_1\cup R_2$, use that $G$ is $\varepsilon n$-joined to find $100$ vertex-disjoint edges in the neighbourhood of $v$, disjoint from all the other choices and disjoint from $\cup_{i\in [3]}R_i$. This is possible as 
        \[d(v,V(G)\setminus \cup_{i\in [3]}R_i)\overset{\eqref{eq:triangle}}{\ge} \frac{\mu n}{4} +200 |R_1 \cup R_2|\ge 2\varepsilon n+200|R_1\cup R_2|.\]
        So, for a vertex $v\in R_1\cup R_2$, at each step we can pick disjoint subsets $X_v, Y_v\subset N(v)\setminus \cup_{i\in [3]}R_i$ of unused vertices such that $|X_v|=|Y_v|=\varepsilon n$, and then pick an edge between $X_v$ and $Y_v$ using that $G$ is $\varepsilon n$-joined.
    \end{claimproof}  

    Our aim now is to suitably glue these absorbers to create robust absorbers.
    
    \stepcounter{propcounter}
	\begin{claim}\label{clm:absorbing-paths}
            There exists an auxiliary bipartite graph $H$ with vertex classes $[3r]$ and $R_1\cup R_2$, and a collection of vertex-disjoint paths $Q_1,\ldots, Q_{3r}$ in $G$  such that $\cup_{i\in [3r]}V(Q_i)$ is disjoint from $\cup_{i\in [3]}R_i$, 
            \begin{enumerate}[label = \upshape\textbf{\Alph{propcounter}\arabic{enumi}}]	
				\item\label{absorber:2} $Q_i$ is an absorbing path of size $1000$ for each $v\in N_H(i)$ and $i\in [3r]$, and
				\item\label{absorber:3} for every subset $R_2'\subset R_2$ of size $|R_2'|=r$, there is a perfect matching in $H$ between $R_1\cup R_2'$ and $[3r]$. 
		\end{enumerate}
	\end{claim}
		
        \begin{claimproof}[Proof of \Cref{clm:absorbing-paths}] 
            Let $H$ be the bipartite graph from~\Cref{lemma:template} with $r$ in the place of $n$, and let $R_1$ and $R_2$ play the role of the sets $Y$ and $Z$, respectively, and let $X=[3r]$. For every vertex $v\in R_1\cup R_2$, let $f_v:N_H(v)\to [100]$ be an injective function listing the neighbours of $v$ in $H$. Letting $Q_0$ be the empty graph, suppose that we have found $Q_0,\ldots, Q_{i-1}$ for some $i\in [3r]$ and let us show how to find $Q_i$. Let $v_1,\ldots, v_q$ be a labelling of $N_H(i)$ and note that $q\le \Delta(H)\le 100$. Using~\ref{property:connecting}, for each $1\le j<q$, find a path $Q'_{ij}$ of length $3$ between $b_{v_j,f_{v_j}(i)}$ and $a_{v_{j+1},f_{v_{j+1}}(i)}$, which is possible because
			\[d(v,V(G)\setminus\cup_{i\in [3]}R_i)\overset{\eqref{eq:triangle}}{\ge}\frac{\mu n}{4}+200|R_1\cup R_2|\ge \frac{\mu n}{8}+2\cdot 10^5|R_1\cup R_2|,\]
			provided $1/n_0\ll \alpha,\theta\ll \mu$. Then, repeatedly using~\ref{property:connecting}, find a path $Q'_{iq}$ starting from $b_{v_q,f_{v_q}(i)}$ and having suitable size so that $\cup_{j\in [q]}Q'_{i,j} \cup \{ a_{v_1,f_1(i)}b_{v_1,f_1(i)}\}$ has exactly 1000 vertices. Again, this is possible due to~\eqref{eq:triangle} and because we are looking for a path with $1000>4q$ vertices.
	\end{claimproof}

        \noindent\textit{Step 3. Finding an almost spanning $K_t$-subdivision.}
        Let $A= \cup_{i\in [3r]}V(Q_i)$ be the set of vertices covered by the absorbing paths we found in Claim~\ref{clm:absorbing-paths}. Note that $|A|=3000r$ as each path has exactly $1000$ vertices. Let $V'=V(G)\setminus (R_1\cup R_2\cup R_3\cup A)$ and pick arbitrarily our branching vertices $u_1,\ldots,u_t \in V'$. To connect the branching vertices, we will find a collection of $\binom{t}{2}$ vertex-disjoint paths covering most of $G[V']$. Indeed, as $G$ is $\varepsilon n$-joined, we can use Lemma~\ref{lemma:long-path} to find a path $P$ in $G'-\{u_1,\ldots, u_t\}$ such that 
		\[|P|\ge |V'|-t-2\varepsilon n.\]
        Divide $P$ into paths $(P'_{e})_{e\in \binom{[t]}{2}}$ so that $\big||P'_{e}|-|P'_{e'}|\big|\le 1$ for all $e,e'\in \binom{[t]}{2}$. Next, for each $e=ij\in \binom{[t]}{2}$, form a $(u_i,u_j)$-path $P_{ij}=P_e$ by using $P'_{ij}$ and connecting through $R_3$, using~\ref{property:connecting}, the endpoints of $P'_{ij}$ with $u_i$ and $u_j$. This is possible, as for all $v\in V(G)$, \ref{thm1:degree:reservoir} gives $d(v,R_3)\ge \mu \alpha n/2\ge c^2n+ 100\varepsilon n\ge t^2n+100\varepsilon n$. Thus, we have found a $K_t$-subdivision, say $S$, such that 
		
        \begin{enumerate}[$\bullet$]
            \item S is nearly-balanced, and
            \item $|V(S)\setminus \cup_{i\in [3]}R_i|\ge n-|R_1\cup R_2\cup R_3|-|A|-2\varepsilon n$, and
		\item\label{thm1:covering:2} $|V(S)\cap R_3|=4\cdot \binom{t}{2}$.

	\end{enumerate}
		
        \noindent\textit{Step 4. Finding a spanning $K_t$-subdivision.} 
        The last step of the proof consists of using the absorbing paths we built in Step 2 to turn the $K_t$-subdivision $S$ into a spanning nearly-balanced subdivision. To do so, we first need to incorporate each of the absorbing paths $(Q_i)_{i\in [3r]}$ into $S$. 

        For $1\le i\le 3r$, pick a path $P_{e}$ of minimal length and take two consecutive vertices $x,y\in V(P_e)$,  and replace the edge $xy\in E(P_e)$ with the $(x,y)$-path which consists of $Q_i$ together with two paths of length $3$ connecting the endpoints of $Q_i$ with $x$ and $y$, respectively, and whose interior vertices belong to $R_3$. This is done, again, using~\ref{thm1:degree:reservoir} and~\ref{property:connecting}, and it is possible as, for every $v\in V(G)$, 
		\[d(v,R_3\setminus V(S)){\ge} \frac{\mu |R_3|}{2}-4\cdot \binom{t}{2}\ge \frac{\mu \alpha n}{2}-2c^2n\ge \frac{\mu\alpha n}{4}+4\cdot 3r,\]
		as $1/n_0\ll c\ll  \theta\ll\alpha\ll\mu$. Thus we have found a $K_t$-subdivision, say $S'$, whose branching path lengths differ by at most $1005$. Indeed, unless $3r$  divides $t\choose2$, then some of the branching paths will contain one more absorbing path $Q_i$ than others. Hence, they may contain at most $1005$ more vertices, $1000$ for the path itself, $2$ to glue each side and one final vertex which is not in $S'$ yet but will be embedded right at the end of the proof. Moreover, $S'$ satisfies $|V(S')\setminus \cup_{i\in [3]}R_i|\ge n-|R_1\cup R_2\cup R_3|-2\varepsilon n$, $|V(S')\cap R_3|=4\cdot \binom{t}{2}+4\cdot 3r$, and
	
        \stepcounter{propcounter}
		\begin{enumerate}[label = \textbf{\Alph{propcounter}}] 
                \item\label{thm1:absorbing} each path $Q_i$ is a subpath of one of the paths from $S'$.
		\end{enumerate}
        
        The last step we need before using the absorbing paths is to include each vertex from $V(G)\setminus (V(S')\cup R_1\cup R_2)$ into the $K_t$-subdivision. To do that, we will connect through $R_2$ now. Observe that the number of vertices used from $R_3$ thus far is exactly 
        $4\cdot {t\choose2}+4\cdot3r$ and that 
        \begin{equation}\label{thm:main:ending}|R_3|=\alpha n\geq4\cdot {t\choose2}+4\cdot3r +1005\cdot {t\choose 2}+2\varepsilon n .
        \end{equation}
        Therefore, by Lemma~\ref{lemma:long-path} we may find a path $P''$ in $V(G)\setminus (V(S')\cup R_1\cup R_2)$ covering all but exactly $2\varepsilon n$ vertices. Furthermore, by~\ref{thm:main:ending} $|P''|\geq 1005 \cdot {t\choose 2}$. 

        Firstly, we embed into $S'$ those vertices in $V(G)\setminus (V(S')\cup R_1\cup R_2)$ which are not covered by $P''$. This is done using~\ref{property:connecting} to make the connections through paths of length $3$ with internal vertices in $R_2$. We do this by embedding one vertex at a time in a shortest branching path, ensuring that the lengths of the branching paths differ by at most $1005$ still. Since $|P''| \geq 1005\cdot {t\choose2}$, we may subdivide the path into $t\choose 2$ segments of appropriate lengths, taking into consideration the fact that some branching paths may contain one more absorbing path $Q_i$ than others, and embed them into the subdivision. This is once again done by using~\ref{property:connecting} and connecting by paths of length 3 with internal vertices in $R_2$. In particular, embedding the $2\varepsilon n$ vertices as well as the $t \choose 2$ segments of $P''$ is possible as, as for every every vertex $v\in V(G)$ we have
		\[d(v,R_2)\overset{\ref{thm1:degree:reservoir}}{\ge} \frac{\mu |R_2|}{2}\ge\frac{\mu\theta n}{4}+4\cdot 2\varepsilon n+4\cdot \tbinom{t}{2},\]
		and at each step, we use exactly 4 new vertices from $R_2$. Note that picking the segment lengths of $P''$ carefully ensures that at this point the branching path sizes differ by at most one. We are now left to embed the remainder of $R_1 \cup R_2$.
        
        Let $R_2'\subset R_2$ be the set of leftover vertices in $R_2$ and note that $|R_2'|=|R_2|-4\cdot{t\choose 2}-8\varepsilon n=r$. Now we are ready to use the property of the absorbing paths. By~\ref{absorber:3}, there is perfect matching, say $M$, between $[3r]$ and $R_1\cup R_2'$. For each edge $iz_i\in E(M)$, by~\ref{absorber:2} we know that $Q_i$ is an absorbing path for $z_i$ and thus, by~\ref{thm1:absorbing} and~\ref{absorber:property1}, we can include $z_i$ into the $K_t$-subdivision. Since $M$ is a perfect matching, we have included each vertex from $R_1\cup R_2'$ and thus found a spanning nearly-balanced $K_t$-subdivision.  
\end{proof}

	\section{Concluding remarks}
    As mentioned in the introduction, the result of McDiarmid and Yolov~\cite{mcdiarmid2017hamilton} gives an optimal dependency between the minimum degree of the graph and the size of the bipartite holes forcing Hamiltonicity. To state this result precisely, we need some definitions. Say a graph $G$ has a $(s,t)$-bipartite-hole if there are disjoint sets $A,B\subset V(G)$ with $|A|=s$, $|B|=t$, and no edges between $A$ and $B$. The \textit{bipartite-hole number} of $G$, which we denoted by $\tilde{\alpha}(G)$, is the smallest integer $m$ that can be written as $m=s+t-1$ for some positive integers $s$ and $t$ so that $G$ has no $(s,t)$-bipartite-hole.

\begin{theorem}[McDiarmid and Yolov~\cite{mcdiarmid2017hamilton}]\label{thm:McDiarmidYolov}If $G$ is a graph with at least 3 vertices and $\delta(G)\ge \tilde{\alpha}(G)$, then $G$ is Hamiltonian.\end{theorem}
The minimum degree condition in Theorem~\ref{thm:McDiarmidYolov} is best possible, as, for any $r\ge 1$, $G=K_{r,r+1}$ has minimum degree $\delta(G)=r$ and bipartite--hole number $\tilde{\alpha}(G)=r+1$ but is not Hamiltonian. Note that in Theorem~\ref{thm:main}, the  $\varepsilon n$-joined condition of $G$ implies that $\tilde{\alpha}(G)\le 2\varepsilon n-1$ and, from the proof of Theorem~\ref{thm:main}, we require that $\varepsilon\ll \mu$, which then implies $\tilde{\alpha}(G)\le 2\varepsilon n-1\le \delta(G)$. It would be interesting to see if Theorem~\ref{thm:McDiarmidYolov} can be extended in the context of Theorem~\ref{thm:main}.

    \begin{question}Is it true that there is a constant $c\in (0,1)$ such that for $n$ large enough the following holds. If $G$ is an $n$-vertex graph with $\delta(G)\ge \tilde{\alpha}(G)$, then $G$ contains a spanning $K_t$-subdivision for all $2\le t\le c\sqrt{n}$. 
    \end{question}

\bibliographystyle{abbrv}
\bibliography{subdivisions.bib}

\end{document}